\documentclass[a4paper,10pt]{amsart}
\usepackage[colorlinks,linkcolor=blue,citecolor=blue]{hyperref}
\usepackage{latexsym, amssymb, amsmath, amsthm, mathrsfs, bbm}
\usepackage[all, knot]{xy}
\xyoption{arc}

\usepackage{anysize}\marginsize{25mm}{25mm}{30mm}{30mm}

\def \To{\longrightarrow}
\def \dim{\operatorname{dim}}

\def \C{\mathcal{C}}

\def \d{\delta}

\def \Z{\mathbb{Z}}

\def \k{\mathbbm{k}}
\def \1{\mathbf{1}}
\def \Id{\operatorname{Id}}
\def \rep{\operatorname{rep}}

\numberwithin{equation}{section}

\newtheorem{theorem}{Theorem}[section]
\newtheorem{lemma}[theorem]{Lemma}
\newtheorem{proposition}[theorem]{Proposition}

\newtheorem{remark}[theorem]{Remark}

\begin{document}

\title[Quasi-Quantum Planes and Quasi-Quantum Groups]{Quasi-Quantum Planes and Quasi-Quantum Groups \\ of Dimension $p^3$ and $p^4$$^\dag$}\thanks{$^\dag$Supported by PCSIRT IRT1264, SRFDP 20130131110001 and SDNSF ZR2013AM022.}

\subjclass[2010]{16T05, 16T20, 16G20}

\keywords{quasi-quantum plane, quasi-quantum group, Hopf quiver}

\author[H.-L. Huang]{Hua-Lin Huang}
\address{School of Mathematics, Shandong University, Jinan 250100, China} \email{hualin@sdu.edu.cn}

\author[Y. Yang]{Yuping Yang*}\thanks{*Corresponding author.}
\address{School of Mathematics, Shandong University, Jinan 250100, China} \email{yupingyang.sdu@gmail.com}

\date{}
\maketitle

\begin{abstract}
The aim of this paper is to contribute more examples and classification results of finite pointed quasi-quantum groups within the quiver framework initiated in \cite{qha1, qha2}. The focus is put on finite dimensional graded Majid algebras generated by group-like elements and two skew-primitive elements which are mutually skew-commutative. Such quasi-quantum groups are associated to quasi-quantum planes in the sense of nonassociative geomertry \cite{m1, m2}. As an application, we obtain an explicit classification of graded pointed Majid algebras with abelian coradical of dimension $p^3$ and $p^4$ for any prime number $p.$
\end{abstract}

\section{Introduction}
The classification problem of finite pointed tensor categories and the underlying quasi-quantum groups in accordance to the Tannaka-Krein duality has been an active research theme for quite some time. In \cite{qha1, qha2}, the quiver framework was proposed by the first author to tackle this problem and some interesting results have been obtained in this direction, see for example \cite{qha3, qha4, lvoz}.

The aim of this paper is to contribute more examples and classification results of finite pointed quasi-quantum groups within the quiver framework. We focus on the class of finite dimensional graded pointed Majid algebras which are generated by group-like elements and two skew-primitive elements which are mutually skew-commutative. The reason is twofold. On the one hand, such Majid algebras are relatively easy and a complete classification may be attainable. On the other hand, this class of Majid algebras are interesting in nonassociative geometry, namely, they may be viewed as the ``coordinate algebra" of nonassociative planes, see \cite{m1, m2}.

The notion of Majid algebras adopted here stands for coquasi-Hopf algebras, or dual quasi-Hopf algebras used by some authors. We remark that the study of Majid algebras and their comodule categories is the main task of the classification problem of tensor categories and quasi-quantum groups, see for instance \cite{qha1} for an explanation.

Throughout we work over the field $\k$ which is algebraically closed with characteristic $0.$ We aim to classify finite dimensional graded pointed Majid algebras $M$ over $\k$ satisfying
\begin{itemize}
\item[(R1):] $M$ is generated by an abelian group $G$ and two skew-primitive elements $\{X,Y\};$ and \\
\item[(R2):] the skew-primitive elements are skew-commutative, i.e., $XY=qYX$ for some $q\in \k^*.$
\end{itemize}
Here, by graded we mean $M$ is coradically graded, that is, $M$ admits a decomposition $M=\bigoplus_{n \ge 0} M(n)$ and $M_n:=\bigoplus_{0 \le i \le n} M(i)$ is the $n$-th term of its coradical filtration. By the assumption, the coradical of $M$ is $M_0=(\k G, \Phi),$ which is a Majid subalgebra of $M$ and the associator $\Phi$ is in fact a normalized 3-cocycle on $G.$ By $_{M_0}^{M_0}\mathcal{Y}\mathcal{D}^{\Phi}$ we denote the Yetter-Drinfeld category of $M_0,$ see \cite{qha2}. Let $\pi: M \rightarrow M_0$ be the natural projection and we consider the associated coinvariant subalgebra $$R:=M^{\operatorname{coinv} M_0}=\{x\in M| (\Id\otimes \pi)\Delta(x)=x\otimes 1\},$$ where $\Delta$ is the coproduct of $M.$ Then $R$ is actually a braided Hopf algebra in $_{M_0}^{M_0}\mathcal{Y}\mathcal{D}^{\Phi}$ generated by primitive elements $\{X,Y\}$ with relations
\begin{gather}
X^{\vec{N}_1}=0, \quad Y^{\vec{N}_2}=0 \ \mathrm{for \ some \ positive \ integers} \ N_1>1, \ N_2>1,\\
XY=qYX \ \mathrm{for \ some} \ q \in \k^*,\ \mathrm{and} \\
\bigtriangleup(X)=X\otimes 1+1\otimes X,\ \ \bigtriangleup(Y)=Y\otimes 1+1\otimes Y
\end{gather} where $X^{\vec{N}}$ means $\underbrace{(\cdots((}_{N-1}XX)X)\cdots X).$ It is worthy to note that in general the multiplication of $R$ is not associative in the usual sense since the associativity constraint of $_{M_0}^{M_0}\mathcal{Y}\mathcal{D}^{\Phi}$ is nontrivial, but quasi-associative in the sense that it is associative up to a 3-cocycle. The previous relations indicate that the algebra $R$ may be viewed as the ``coordinate algebra" of some ``plane" in the braided category $_{M_0}^{M_0}\mathcal{Y}\mathcal{D}^{\Phi}$ according to the philosophy of noncommutative algebraic geometry \cite{manin} and nonassociative geometry \cite{m1, m2}. Therefore, we adopt the notion quasi-quantum plane for $R$ in accordance with the terminology of \cite{m1, m2, manin}. It is well known that one may recover $M$ by $R$ via a quasi-version (see e.g. \cite{ap}) of the procedure of Majid's bosonization \cite{m}, hence the study of $M$ and $R$ is essentially identical.

In this paper we mainly work on $M,$ as in this situation the handy Hopf quiver \cite{cr} technique can be applied. Recall that, a Hopf quiver is defined through the so-called ramification data of groups. For a group $G,$ let $\C$ be its conjugacy classes, a ramification datum of $G$ is a formal sum $R=\sum_{C\in \C}R_CC$ with nonnegative integer coefficients. The Hopf quiver associated to the ramification datum $(G,R),$ denoted by $Q(G,R),$ is the quiver with vertices the elements of $G$ and has $R_C$ arrows from $x$ to $cx$ for each $x \in G$ and $c \in C.$ By the Gabriel-type theorem for pointed Majid algebras \cite{qha1}, $M$ can be realized as a subalgebra of a quiver Majid algebra on some unique Hopf quiver. Within this quiver framework, we observe that finite dimensional graded pointed Majid algebras satisfying (R1-R2) live on Hopf quivers of the form $Q(\langle g,h \rangle,g+h)$ where $\langle g,h \rangle$ is a finite abelian group generated by two elements $g$ and $h,$ and the two skew-primitive elements are nothing but the two arrows $X:1 \rightarrow g, \ \ Y:1 \rightarrow h$ where $1$ is the unit of the group $\langle g,h \rangle.$ Then the conditions of $M$ may be interpreted by combinatorics of Hopf quivers and further may be determined by some projective representations of $G$ on the $\k$-space spanned by $\{X, Y\}$ according to \cite{qha2}. This allows us to give a complete classification of such Majid algebras and to classify some quasi-quantum groups of low dimension. Our classification list of Majid algebras recovers some examples, in a dual form, obtained in the interesting paper \cite{a}  which appear as basic quasi-Hopf algebras over cyclic groups, and the list also contains new examples which are over non-cyclic groups.

This short paper is organized as follows. In Section 2, we compute all the graded Majid algebra structures on $\k Q(\langle g,h \rangle,g+h)$ and determine all the Majid subalgebras associated to quasi-quantum planes. On this base, in Section 3 we give the classification of graded pointed Majid algebras with abelian coradical of dimension $p^3$ and $p^4$ for any prime number $p.$

\section{Quasi-quantum planes and the associated pointed Majid algebras}
The main task of this section is to classify graded pointed Majid algebras associated to quasi-quantum planes. The quiver framework is briefly recalled at first and the reader is referred to \cite{qha1, qha2} for more unexplained details.

\subsection{Majid bimodules}
Let $G$ be a finite group and $\Phi$ a normalized 3-cocycle on $G.$ Let $\k G$ be the group algebra with the usual diagonal coproduct. Extend $\Phi,$ without changing the notation, by linearity to a function on $(kG)^{\otimes 3},$ then $(\k G, \Phi)$ becomes a Majid algebra with associator $\Phi$ and antipode $(S,\alpha, \beta)$ given by $S(g)=g^{-1}, \ \alpha(g)=1$ and $\beta(g)=\frac{1}{\Phi(g,g^{-1},g)}$ for any $g\in G.$ By definition, a $(\k G, \Phi)$-Majid bimodule $M$ is a $\k G$-bicomodule, or equivalently a $G$-bigraded space $M=\bigoplus_{g,h \in G} \ ^gM^h$ with $(g,h)$-isotypic component
$$^gM^h=\{ m \in M \ | \ \d_{_L}(m)=g \otimes m, \ \d_{_R}(m)=m \otimes h \} $$ endowed with a compatible quasi-bimodule structure satisfying the
following equalities:
\begin{gather}
e.(f.m)=\frac{\Phi(e,f,g)}{\Phi(e,f,h)}(ef).m,\\
(m.e).f=\frac{\Phi(h,e,f)}{\Phi(g,e,f)}m.(ef),\\
(e.m).f=\frac{\Phi(e,h,f)}{\Phi(e,g,f)}e.(m.f),
\end{gather}
for all $e,f,g,h \in G$ and $m \in \ ^gM^h.$ Recall that by compatible is meant the quasi-bimodule structure maps are bicomodule morphisms. 

\subsection{Projective representations}
Denote the category of $(\k G,\Phi)$-Majid bimodules by $\mathcal{M}\mathcal{B}(\k G,\Phi).$ It was showed in \cite{qha2} that the category $\mathcal{M}\mathcal{B}(\k G,\Phi)$ is equivalent to the product of projective representation categories of some subgroups of $G.$ We recall this equivalence for the case of $G$ being a finite abelian group which is enough for our purpose. For each $V \in \mathcal{M}\mathcal{B}(\k G,\Phi)$, let $V=\oplus_{g,h\in G} \ ^gV^h$ be the decomposition of isotypic components. According to the axioms of $(\k G,\Phi)$-Majid bimodules, we have for all $f,g,h\in G,$
\begin{equation*}
f .  ^gV^h=\ ^{fg}V^{fh},\ \ ^gV^h . f =\ ^{gf}V^{hf}.
\end{equation*}
Since $G$ is abelian, we have
\begin{equation}
(f . ^gV^h) . f^{-1}=\  ^gV^h.
\end{equation}
For each $v\in \ ^gV^1,$ define $g\triangleright v=(g . v). g^{-1}.$ It is easy to verify that
\begin{equation}
1\triangleright v=v, \ \ e\triangleright (f\triangleright v)=\widetilde{\Phi}(e,f,g)(ef)\triangleright V,
\end{equation}
where 
\begin{equation}\widetilde{\Phi}(e,f,g)=\frac{\Phi(e,f,g)\Phi(ef,f^{-1},e^{-1})\Phi(e,fg,f^{-1})}{\Phi(efg,f^{-1},e^{-1})\Phi(e,f,f^{-1})}.
\end{equation}
  Denote by
\begin{equation}
\widetilde{\Phi}_g: G\times G\To \k^*, \ \ (e,f)\to \widetilde{\Phi}(e,f,g)
\end{equation}
the mapping induced by $\widetilde{\Phi}.$ Then $\widetilde{\Phi}_g$ is a 2-cocycle on $G$ and hence $^gV^1$ is a projective representation of $G.$ By $\rep( G, \widetilde{\Phi}_g)$ we denote the category of projective representations of $G$ associated to the 2-cocycle $\widetilde{\Phi}_g.$ The we have the following category equivalence $$\mathcal{M}\mathcal{B}(\k G,\Phi)\cong \prod_{g\in G} \rep( G, \widetilde{\Phi}_g).$$

\subsection{Quantum shuffle product}
Let $G$ be a group, $R$ a ramification datum, and $Q(G,R)$ the associated Hopf quiver. By $Q_l$ we denote the set of paths of $Q(G,R)$ of length $l.$ It is obvious that $Q_0=G$ and $Q_1$ is the set of arrows. It was proved in \cite{qha1} that the set of graded Majid algebra structures on the path coalgebra $\k Q(G,R)$ is in one-to-one correspondence with the set of $(\k G,\Phi)$-Majid bimodule structures on $\k Q_1$.

For completeness we recall the quantum shuffle product procedure \cite{qha1} which turns a $(\k G,\Phi)$-Majid bimodule structure on $\k Q_1$ into a graded Majid algebra structure on $\k Q(G,R).$ Suppose that $Q(G,R)$ and a necessary $(\k G, \Phi)$-Majid bimodule structure on $\k Q_1$ are given. Let $p \in Q_l$ be a path. An $n$-thin split of it is a sequence $(p_1, \ \cdots, \ p_n)$ of vertices and arrows such that the concatenation $p_n \cdots p_1$ is exactly $p.$ These $n$-thin splits are in one-to-one correspondence with the $n$-sequences of $(n-l)$ 0's and $l$ 1's. Denote the set of such sequences by $D_l^n.$ Clearly $|D_l^n|={n \choose l}.$ For $d=(d_1, \ \cdots, \ d_n) \in D_l^n,$ the corresponding $n$-thin split is written as $dp=((dp)_1, \ \cdots, \ (dp)_n),$ in which $(dp)_i$ is a vertex if $d_i=0$ and an arrow if $d_i=1.$ Let $\alpha=a_m \cdots a_1$ and $\beta=b_n \cdots b_1$ be paths of length $m$ and $n$ respectively. For $d \in D_m^{m+n}$, let $\bar{d} \in D_n^{m+n}$ be the complement sequence of $d$ which is obtained from $d$ by replacing each 0 by 1 and each 1 by 0. Define an element
\begin{equation}
(\alpha \ast \beta)_d=[(d\alpha)_{m+n}.(\bar{d}\beta)_{m+n}] \cdots
[(d\alpha)_1.(\bar{d}\beta)_1]
 \end{equation}
 in $\k Q_{m+n},$ where $[(d\alpha)_i.(\bar{d}\beta)_i]$ is understood as the action of $(\k G, \Phi)$-Majid bimodule on $\k Q_1$ and these terms in different brackets are put together by cotensor product, or equivalently concatenation. The reader is referred to \cite{qha1} for detail. In terms of these notations, the formula of the quantum shuffle
product of $\alpha$ and $\beta$ is given as follows:
\begin{equation}
\alpha \ast \beta=\sum_{d \in D_m^{m+n}}(\alpha \ast \beta)_d \ .
\end{equation}

\subsection{Projective $\k G$-representations on $\k \{X,Y\}$} In the rest of this section, let $G$ be a finite abelian group generated by $g$ and $h$ and let $Q$ denote the Hopf quiver $Q(G,g+h).$ The two arrows $1\to g$ and $1\to h$ of $Q$ with source $1$ (the unit of $G$) are denoted by $X$ and $Y$ respectively. According to the possible structure of the group $G,$ we need to consider two cases: 1, $G=\Z_m\times \Z_n$ for some $m,n>1;$ 2, $G=\Z_m=\langle e\rangle$ and $g=e^\alpha, \ h=e^\beta$ for some $0<\alpha<m, \ 0<\beta<n$ such that $(\alpha,\beta)=1.$ In the first  case, it is clear that $mn\leq |g||h|.$ Here $|g|$ stands for the order of $g.$ If $mn=|g||h|,$ then $G=\langle g\rangle \times \langle h\rangle$; if $mn<|g||h|,$ then $g^s=h^t$ for some $0<s<|g|, \ 0<t<|h|.$ In the latter, the group $G$ may be seen as the quotient group $\langle g\rangle \times \langle h\rangle/\langle g^s h^{-t}\rangle$ and the associated pointed Majid algebras over $G$ can be obtained as those over $\langle g\rangle \times \langle h\rangle$ modulo the Majid ideal generated by $(g^s - h^t).$ Hence, in the following we consider mainly the case of $G=\langle g\rangle \times \langle h\rangle.$

With the above preparation, now we are ready to work on the Majid algebras associated to quasi-quantum planes. As suggested by the quiver classification project \cite{qha1, qha2}, the first step is to compute the projective representations of $\k G$ on $\k \{X,Y\},$ the $\k$-space spanned by $X$ and $Y;$ the second step is to recover $(\k G, \Phi)$-Majid bimodules on $\k Q_1$ from the obtained projective representations; the third step is to compute the Majid algebra structures satisfying conditions (R1-R2) using quantum shuffle product.

We begin with some notations and facts about normalized 3-cocycles on $G$ which will act as the associators of our Majid algebras later on. Let $\zeta_n$ be a primitive $n$-th root of unity. By $[x]$ we denote the integer part of a rational number $x$ and $(m,n)$ the greatest common divisor of two integers $m$ and $n.$ It was showed in \cite{bgrc1} that a complete list of representatives of normalized 3-cocycles on $G =\Z_m\times \Z_n =\langle g\rangle \times \langle h\rangle$ is given by
\begin{equation}
\Phi_{a,b,c}(g^ih^j,g^sh^t,g^kh^l)=\zeta_m^{a[\frac{k+s}{m}]i}\zeta_n^{b[\frac{k+s}{m}]j}\zeta_n^{c[\frac{t+l}{n}]j}, \quad \forall 0\leq i,s,k< m, \ 0\leq j,t,l< n
\end{equation}
with $0\leq a < m, \ 0\leq b < (m,n), \ 0\leq c < n.$ The normalized 3-cocycles on a cyclic group can be seen as (2.10) with $j,t,l=0.$ By direct computation the explicit form of (2.6) may be given by
\begin{equation}
\widetilde{\Phi}_{a,b,c}(g^ih^j,g^sh^t,g^kh^l)=\zeta_m^{-a[\frac{(m-i)'+(m-s)'}{m}]k}\zeta_n^{-b[\frac{(m-i)'+(m-s)'}{m}]l}\zeta_n^{-c[\frac{(n-j)''+(n-t)''}{n}]l}
\end{equation}
for $G=\Z_m\times \Z_n$ and by
\begin{equation}
\widetilde{\Phi}_{a}(g^i,g^j,g^k)=\zeta_m^{-a[\frac{(m-i)'+(m-j)'}{m}]k}
\end{equation}
for $G=\Z_m.$ Here $i'$ (resp. $i''$) is the remainder of the division of $i$ by $m$ (resp. $n$).

For our purpose, we need to determine projective $\k G$-representations on $\k \{X, Y\}$ with respect to 2-cocycles induced by $\widetilde{\Phi}.$
In the case with $g \ne h,$ we need to compute $(\k G, \widetilde{\Phi}_g)$-representations on $\k X$ and $(\k G, \widetilde{\Phi}_h)$-representations on $\k Y.$ In the case with $g=h,$ note that $G=\langle g\rangle$ and $\widetilde{\Phi}_g$ is symmetric, i.e. $\widetilde{\Phi}_g(x,y)=\widetilde{\Phi}_g(y,x)$ for all $x,y\in G,$ so the $\widetilde{\Phi}_g$-twisted group algebra of $G$ is commutative and then $\k \{X,Y\}$ is a direct sum of two one-dimensional projective representations of $G.$ By \cite{qha1}, we may also assume without loss of generality that $\k X$ and $\k Y$ are one-dimensional projective representations.

\begin{lemma}
\begin{itemize}
\item[1.] Write $\Phi=\Phi_{a,b,c}$ and $G=\Z_m\times \Z_n=\langle g\rangle \times \langle h\rangle$ for brevity. The following equations $$g\triangleright X=\lambda_1 X, \quad h\triangleright X=\lambda_2 X \quad (\mathrm{resp}. \ g\triangleright Y=\eta_1 Y, \quad h\triangleright Y=\eta_2 Y)$$ define a $(G,\widetilde{\Phi}_{g})$-projective (resp. $(G,\widetilde{\Phi}_{h})$-projective) representation on $\k X$ (resp. $\k Y$) if and only if $\lambda_1^m=\zeta_m^a, \ \lambda_2^n=1$ (resp. $\eta_1^m=\zeta_n^b,\ \eta_2^n=\zeta_n^c$).
\item[2.] Again write $\Phi=\Phi_a$ and $G=\Z_m=\langle e \rangle$ for brevity. The following equation $$e\triangleright X=\mu_1 X \quad (\mathrm{resp}. \ e \triangleright Y=\mu_2 Y)$$ defines a $(G,\widetilde{\Phi}_{g})$-projective (resp. $(G,\widetilde{\Phi}_{h})$-projective) representation on $\k X$ (resp. $\k Y$) if and only if $\mu_1^m=\zeta_m^{a\alpha}$ (resp. $\mu_2^m=\zeta_m^{a\beta}$).
\end{itemize}
\end{lemma}

\begin{proof}
We only prove the case in 1 for $\k X$ as the other cases can be proved similarly. Write the product of the twisted group algebra $\k^{\widetilde{\Phi}_{g}}G$ by $\star.$ By induction on $i$ it is very easy to get
\begin{equation}
g^{\star i}=\zeta_m^{-a(i-1)}g^i, \quad h^{\star i}=h^i.
\end{equation}
Then the claim follows by \[ \lambda_1^m X=g^{\star m} \triangleright X = \zeta_m^{-a(m-1)}g^m \triangleright X = \zeta_m^a X, \quad \lambda_2^n X=h^{\star n} \triangleright X = h^n \triangleright X = X.\]
\end{proof}

\subsection{The associated $(\k G,\Phi)$-Majid bimodules}
Keep the notations of Lemma 2.1. Now we recover the $(\k G,\Phi)$-Majid bimodule structures on $\k Q_1$ associated to the obtained projective representations in the sense of \cite{qha2}. For the case with $G=\langle g \rangle \times \langle h \rangle=\Z_m\times \Z_n,$ let $\Gamma_{(i,j)}^1$ and $\Gamma_{(i,j)}^2$ denote respectively the arrows $g^ih^j\to g^{i+1}h^j$ and $g^ih^j\to g^ih^{j+1}$ in $Q=Q(G,g+h).$ For the case with $G=<e>=\Z_m,$ let $X_i^1$ and $Y_i^1$ denote respectively the arrows $e^i\to e^{i+\alpha}$ and $e^i\to e^{i+\beta}$ in $Q=Q(G, g+h).$ Note that the bicomodule structure $(\delta_L, \delta_R)$ of $\k Q_1$ is defined according to the quiver structure. Namely, for the case with $G=\Z_m\times \Z_n,$
\begin{eqnarray}
\delta_L(\Gamma_{(i,j)}^1)=g^{i+1}h^j\otimes \Gamma_{(i,j)}^1, & \delta_R(\Gamma_{(i,j)}^1)=\Gamma_{(i,j)}^1\otimes g^ih^j,\\
\delta_L(\Gamma_{(i,j)}^2)=g^ih^{j+1}\otimes \Gamma_{(i,j)}^2, & \delta_R(\Gamma_{(i,j)}^2)=\Gamma_{(i,j)}^2\otimes g^ih^j,
\end{eqnarray}
and for case with $G=\Z_m,$
\begin{eqnarray}
\delta_L(X_i^1)=e^{i+\alpha}\otimes X_i^1, & \delta_R(X_i^1)=X_i^1\otimes e^i,\\
\delta_L(Y_i^1)=e^{i+\beta}\otimes Y_i^1, & \delta_R(Y_i^1)=Y_i^1\otimes e^i.
\end{eqnarray}
For quasi-bimodule structure, there is no harm to assume that
\begin{equation}
g^ih^j\cdot\Gamma_{(0,0)}^1=\Gamma_{(i,j)}^1,\ \ g^ih^j\cdot\Gamma_{(0,0)}^2=\Gamma_{(i,j)}^2
\end{equation}
for $G=\Z_m\times \Z_n$ and
\begin{equation}
e^i\cdot X_0^1=X_i^1,\ \ e^i \cdot Y_0^1=Y_i^1
\end{equation}
for $G=\Z_m.$ 
To make the notations consistent with those given in Subsection 2.4 and for the convenience of the exposition, in the following we write $X=\Gamma_{(0,0)}^1$ and $Y=\Gamma_{(0,0)}^2$ in case $G=\Z_m\times \Z_n,$ and $X=X_0^1$ and $Y=Y_0^1$ in case $G=\Z_m.$ 
The following proposition can be verified by routine computations.

\begin{proposition}
\begin{itemize}
\item[1.] The following equalities
\begin{eqnarray}
 g^ih^j\cdot \Gamma_{(s,t)}^1 &=& \zeta_m^{a[\frac{s+1}{m}]i}\zeta_n^{b[\frac{s+1}{m}]j}\Gamma_{(i+s,j+t)}^1, \\
 g^ih^j\cdot \Gamma_{(s,t)}^2 &=& \zeta_n^{c[\frac{t+1}{n}]j}\Gamma_{(i+s,j+t)}^2,\\
 \Gamma_{(i,j)}^1\cdot g^sh^t &=& (\zeta_m^a\lambda_1)^{-s}\lambda_2^{-t}\Gamma_{(i+j,s+t)}^1,\\
 \Gamma_{(i,j)}^2\cdot g^sh^t &=& (\zeta_n^b\eta_1)^{-s}(\zeta_n^c\eta_2)^{-t}\Gamma_{(i+j,s+t)}^2,
\end{eqnarray} for all $0 \leq i, s \leq m-1, \ 0 \leq j,t\leq n-1$
together with $(2.14-2.15)$ define a $(\k \Z_m\times \Z_n,\Phi_{a,b,c})$-Majid bimodule structure on $\k Q_1(\Z_m\times \Z_n,g+h)$ associated to the projective representation given in 1 of Lemma 2.1.

\item[ 2.] The following equalities
\begin{eqnarray}
e^i\cdot X_j^1=\zeta_m^{a[\frac{j+\alpha}{m}]i}X_{(i+j)}^1, \ \ e^i\cdot Y_j^1=\zeta_m^{a[\frac{j+\beta}{m}]i}Y_{(i+j)}^1 ,\\
X_i^1\cdot e^j=(\zeta_m^{a\alpha}\mu_1)^{-j}X_{(i+j)}^1, \ \  Y_i^1\cdot e^j=(\zeta_m^{a\beta}\mu_2)^{-j}Y_{(i+j)}^1,
\end{eqnarray} for all $0 \leq i,j \leq m-1$
together with $(2.16-2.17)$ define a $(\k\Z_m,\Phi_{a})$-Majid bimodule structure on $\k Q_1(\Z_m,g+h)$ associated to the projective representation given in 2 of Lemma 2.1.
\end{itemize}
\end{proposition}

\subsection{Classification results}
Now we are in the position to give the classification of the pointed Majid algebras satisfying (R1-R2). We need some  notations of quantum binomial coefficients in what follows. For any $\hbar \in \k$, define $l_\hbar=1+\hbar+\cdots +\hbar^{l-1}$ and $l!_\hbar=1_\hbar \cdots l_\hbar$. The Gaussian binomial coefficient is defined by $\binom{l+m}{l}_\hbar:=\frac{(l+m)!_\hbar}{l!_\hbar m!_\hbar}$.
\begin{lemma}
\begin{itemize}
\item[1.] Suppose $\k Q_1(\Z_m\times \Z_n,g+h)$ is endowed with the $(\k \Z_m\times \Z_n,\Phi_{a,b,c})$-Majid bimodule structure as given in 1 of Proposition 2.2. Then in the corresponding quiver Majid algebra $\k Q(\Z_m\times \Z_n,g+h)$ we have
\begin{eqnarray}
X^{\ast \vec{l}}&=&l!_{\zeta_m^{-a}\lambda_1^{-1}}[g^{l-1}\cdot X] [g^{l-2}\cdot X]\cdots [g \cdot X ]X ,\\
Y^{\ast \vec{l}}&=&l!_{\zeta_n^{-c}\eta_2^{-1}}[h^{l-1}\cdot Y][h^{l-2}\cdot Y]\cdots [ h \cdot Y]Y.
\end{eqnarray}
\item[2.] Suppose $\k Q_1(\Z_m,g+h)$ is endowed with the $(\k \Z_m,\Phi_{a})$-Majid bimodule structure as given in 2 of Proposition 2.2. Then in the corresponding quiver Majid algebra $\k Q(\Z_m,g+h)$ we have
\begin{eqnarray}
X^{\ast \vec{l}}&=&l!_{\zeta_m^{-a\alpha^2}\mu_1^{-\alpha}}[g^{l-1}\cdot X] [g^{l-2}\cdot X]\cdots [g \cdot X ]X ,\\
Y^{\ast \vec{l}}&=&l!_{\zeta_m^{-a\beta^2}\mu_2^{-\beta}}[h^{l-1}\cdot Y][h^{l-2}\cdot Y]\cdots [h \cdot Y ]Y.
\end{eqnarray}
\end{itemize}
\end{lemma}
\begin{proof}
By (2.20) and (2.22), we have
\begin{equation*}
g\cdot X = \zeta_m^{a}\lambda_1 X\cdot g.
\end{equation*}
We will use  induction on $l$ to prove the identity. If $l=2$, then we have
\begin{equation*}
\begin{split}
X\ast X &=[g\cdot X][X\cdot 1]+[X\cdot g][1\cdot X]\\
     &=(1+\zeta_m^{-a}\lambda_1^{-1})[g \cdot X ]X.
\end{split}
\end{equation*}
For $l-1$, assume that $X^{\ast \overrightarrow{l-1}}=(l-1)!_{\zeta_m^{-a}\lambda_1^{-1}}[g^{l-2}\cdot X] [g^{l-3}\cdot X]\cdots [g \cdot X ]X .$ Because $g^i\cdot X $ is a scalar multiple of arrow $g^i\to g^{i+1}$, hence $X^{\ast \overrightarrow{l-1}}$ is a scalar  multiple of the path $1\to g\to \cdots \to g^{l-1}.$ Hence
\begin{equation*}
\begin{split}
X^{\ast \vec{l}}=& X^{\ast \overrightarrow{l-1}}\ast X=(l-1)!_{\zeta_m^a\lambda_1}([g^{l-2}\cdot X] [g^{l-3}\cdot X]\cdots [g \cdot X ]X )\ast X \\
=&(l-1)!_{\zeta_m^a\lambda_1}([g^{l-1}\cdot X][g^{l-2}\cdot X] [g^{l-3}\cdot X]\cdots [g \cdot X ]X \\
&+ [(g^{l-2}\cdot X)\cdot g][g^{l-2}\cdot X][g^{l-3}\cdot X]\cdots  [g \cdot X ]X +\cdots \\
&+ [(g^{l-2}\cdot X)\cdot g][(g^{l-3}\cdot X)\cdot g]\cdots  [( g\cdot X)\cdot g][X\cdot g]X)\\
=&(l-1)!_{\zeta_m^{-a}\lambda_1^{-1}}(1+\zeta_m^{-a}\lambda_1^{-1}+ \cdots + \zeta_m^{-(l-1)a}\lambda_1^{-(l-1)}) [g^{l-1}\cdot X] [g^{l-2}\cdot X]\cdots [g\cdot X ]X\\
=&l!_{\zeta_m^{-a}\lambda_1^{-1}}[g^{l-1}\cdot X] [g^{l-2}\cdot X]\cdots [g \cdot X ]X.
\end{split}
\end{equation*}
So (2.26) is proved. The other equations (2.27-2.29) can be proved in a similar manner.
\end{proof}

\begin{proposition}
\begin{itemize}
\item[1.] $\k Q(\Z_m\times \Z_n,g+h)$ admits a graded Majid algebra structure with associator $\Phi_{a,b,c}$ such that the Majid subalgebra generated by $\{X,Y\}$ and the vertex group satisfies (R1-R2) if and only if \begin{equation}mx+ny+(m+1)b\equiv 0 \ \mod \ mn \end{equation} has a solution $(x,y)\in \Z\times\Z$ and $\lambda_1\neq 1, \ \eta_2\neq 1.$
\item[2.]
$\k Q(\Z_m,e^\alpha+e^\beta)$ admits a graded Majid algebra structure with associator $\Phi_{a}$ such that the Majid subalgebra generated by $\{X,Y\}$ and the vertex group satisfies (R1-R2) if and only if \begin{equation}m(\beta x+\alpha y+2a\alpha\beta)+ 2a\alpha\beta\equiv 0 \ \mod \ m^2 \end{equation} has a solution $(x,y)\in \Z\times \Z$ and $\mu_1\neq 1, \ \mu_2\neq 1.$
\end{itemize}
\end{proposition}
\begin{proof}
We only prove case 1 as the proof for case 2 is similar and much easier. Suppose that the Majid subalgebra of $\k Q(\Z_m \times \Z_n,g+h)$ generated by $\{X,Y\}$ and the vertex group satisfies (R1-R2). The condition $X\ast Y=qY\ast X$ for some $q\in \k^*$ is equivalent to
$$[X\cdot h]Y+ [g\cdot Y]X=q([Y\cdot g]X+[h\cdot X]Y).$$ This implies that $h\cdot X=q^{-1}X\cdot h$ and $g\cdot Y= q Y\cdot g.$ Recall that \cite{qha2} we can define $f \triangleright X =( f \cdot X) \cdot f^{-1}, \ \forall f \in G$ and this makes $\k X$ a $(G,\widetilde{\Phi}_{g})$-projective representation. Similarly we may get a projective representation structure on $\k Y.$ On the other hand, with the notations of Lemma 2.1, one has
\begin{equation*}
h\cdot X=(h\cdot X)(h^{-1}h)=\frac{\Phi_{a,b,c}(gh,h^{-1},h)}{\Phi_{a,b,c}(h,h^{-1},h)}(h\triangleright X)\cdot h=\lambda_2 X\cdot h,
\end{equation*} and
\begin{equation*}
g\cdot Y=(g\cdot Y)(g^{-1}g)=\frac{\Phi_{a,b,c}(gh,g^{-1},g)}{\Phi_{a,b,c}(g,g^{-1},g)}(g\triangleright Y)\cdot g=\zeta_n^b \eta_1 Y\cdot g.
\end{equation*}  These force $q^{-1}=\lambda_2$ and  $q=\zeta_n^b \eta_1.$ Now we may apply Lemma 2.1 to get $\lambda_2=\zeta_n^\beta$ and $\eta_1=\zeta_{mn}^b\zeta_m^\gamma$ for some $0\leq \beta \leq n-1, \ 0\leq \gamma \leq m-1.$ Here $\zeta_{mn}$ is an $m$-th root of $\zeta_n.$  Then we have $$1=q^{-1}q=\lambda_2\zeta_n^b\eta_1=\zeta_n^\beta\zeta_n^b\zeta_{mn}^b\zeta_m^{\gamma},$$  which is equivalent to saying
\begin{equation*}m\beta+n\gamma+(m+1)b \equiv 0 \ \mod \ mn. \end{equation*}
By Lemma 2.3, $X^{\ast \vec{N}_1}=0$ and $Y^{\ast \vec{N}_2}=0$ for $N_1 = |\zeta_m^a\lambda_1|, \ N_2=|\zeta_n^c\eta_2|.$ Here $| \cdot |$ means the multiplicative order of a root of unity. The condition of $N_1>1, \ N_2>1$ forces $\zeta_m^a\lambda_1 \neq 1, \ \zeta_n^c\eta_2\neq 1.$ If $a=0,$ then clearly $\lambda_1 \ne 1.$ If $a \ne 0,$ then by Lemma 2.1, $\lambda_1^m=\zeta_m^a,$ and the condition becomes $\lambda_1^{m+1}\neq 1.$ Note that $\lambda_1$ is an $m^2$-th root of unity, and $\lambda_1^{m+1} \neq 1$ is equivalent to $\lambda_1 \ne 1$ as $(m^2, m+1)=1.$ Similarly we have the condition $\eta_2 \neq 1.$

Conversely, if equation (2.30) has a solution in $\Z\times \Z$ and assume that $\lambda_1 \ne 1, \ \eta_2\neq 1$ are roots of unity, then we can choose a quadruple $(\lambda_1,\lambda_2,\eta_1,\eta_2)$ satisfying the conditions of Lemma 2.1 to define projective representations of $\Z_m \times \Z_n$ and get the associated $(\k G, \Phi)$-Majid bimodule as in Proposition 2.2. Then by Lemma 2.3 and a direct computation of $X * Y$ and $Y * X$ as above, it is clear that the corresponding Majid subalgebra satisfies (R1-R2).
\end{proof}

In what follows, we denote the sets of all integer solutions of (2.30) and (2.31) by $A$ and $B$ respectively. Let $\zeta_{(m,n)m}$ be an $m$-th root of $\zeta_{(m,n)},$ and $\zeta_{m^2}$ an $m$-th root of $\zeta_{m}.$  The previous proposition implies the set of graded pointed Majid algebras satisfying (R1-R2) are in one-to-one correspondence to
\begin{equation}\widetilde{A}= \{(\lambda_1,\lambda_2,\eta_1,\eta_2)| \lambda_1^m=\zeta_m^a,\ \lambda_2=\zeta_n^x ,\  \eta_1=\zeta_m^y \zeta_{mn}^b,\  \eta_2^n=\zeta_n^c; \ \lambda_1 \ne 1, \ \eta_2\neq 1; (x,y)\in A\}
\end{equation}
 if $G=\langle g\rangle \times \langle h\rangle= \Z_m\times \Z_n,$ or to
\begin{equation}\widetilde{B}=\{(\mu_1,\mu_2)|\mu_1=\zeta_{m^2}^{a\alpha}\zeta_m^x,\  \mu_2=\zeta_{m^2}^{a\beta}\zeta_m^y;\  \mu_1^{m+\alpha}\neq 1, \ \mu_2^{m+\beta}\neq 1; (x,y)\in B \}
\end{equation}
if $G=\langle e\rangle=\Z_m.$

Combining Propositions 2.2 and 2.4, we have the following classification results.

\begin{theorem}
\begin{itemize}
\item[1.] Given a quadruple $(\lambda_1,\lambda_2,\eta_1,\eta_2) \in \widetilde{A},$ one can associate to it a graded pointed Majd algebra $M(\lambda_1,\lambda_2,\eta_1,\eta_2)$ with generators a group $\langle g\rangle \times \langle h\rangle = \Z_m\times \Z_n$ and two skew-primitive elements $\{X,Y\}$ subject to relations
\begin{gather}
\Delta(X)=X\otimes 1+ g\otimes X,\ \ \ \ \ \Delta(Y)=Y\otimes 1+ h\otimes Y, \\
gX=\zeta_m^a\lambda_1Xg,\  hX=\lambda_2Xh,\  gY=\zeta_n^b\eta_1Yg,\  hY=\zeta_n^c\eta_2Yh, \\
X^{\vec{N}_1}=0,\ Y^{\vec{N}_2}=0, \ XY=\lambda_2^{-1}YX, \ where \ N_1=|\zeta_m^{a}\lambda_1|, \ N_2=|\zeta_n^c\eta_2|,
\end{gather}
and with associator concentrated at degree zero given by
\begin{equation}
\Phi_{a,b,c}(g^ih^j,g^sh^t,g^kh^l)=\zeta_m^{a[\frac{k+s}{m}]i}\zeta_n^{b[\frac{k+s}{m}]j}\zeta_n^{c[\frac{t+l}{n}]j}
\end{equation}
for $0\leq a<m,$ $0\leq b<(m,n)$ and $0\leq c<n$ such that at least one of $\{a,b,c\}$ is nonzero.

Any finite dimensional graded pointed Majid algebra $M$ with $M_0=\k \langle g\rangle \times \langle h\rangle=\k \Z_m\times \Z_n$ and satisfies (R1-R2) such that $\Delta(X)=X\otimes 1+ g\otimes X,$ and $\Delta(Y)=Y\otimes 1+ h\otimes Y$ must be twist equivalent to one of the $M(\lambda_1,\lambda_2,\eta_1,\eta_2)$ described above.
\item[2.]
Given a pair $(\mu_1,\mu_2) \in \widetilde{B},$ one can associate to it a graded pointed Majd algebra $M(\mu_1,\mu_2)$ with generators a group $<e>=\Z_m$ and two skew-primitive elements $\{X,Y\}$ subject to relations
\begin{gather}
\Delta(X)=X\otimes 1+ e^\alpha\otimes X,\quad \Delta(Y)=Y\otimes 1+ e^\beta \otimes Y, \\
eX=\zeta_m^{a\alpha}\mu_1Xe,\ \  eY=\zeta_m^{a\beta}\mu_2Ye, \\
X^{\vec{N}_1}=0,\ Y^{\vec{N}_2}=0, \ XY=\zeta_m^{a\alpha\beta}\mu_2^\alpha YX, \ where \ N_1=|\zeta_m^{a\alpha^2}\mu_1^{\alpha}|, \ N_2=|\zeta_m^{a\beta^2}\mu_2^\beta |,
\end{gather}
and with associator concentrated at degree zero given by
\begin{equation}
\Phi_{a}(e^i,e^j,e^k)=\zeta_m^{a[\frac{k+j}{m}]i}
\end{equation}
for $1\leq a<m.$

Any finite dimensional graded pointed Majid algebra $M$ with $M_0=\k \langle e \rangle=\k \Z_m$ and satisfies (R1-R2) such that $\Delta(X)=X\otimes 1+ e^\alpha\otimes X,$ and $\Delta(Y)=Y\otimes 1+ e^\beta \otimes Y$ must be twist equivalent to one of the $M(\mu_1,\mu_2)$ described above.
\end{itemize}
\end{theorem}

\begin{remark}
We give a brief description for the case in which the group is not the direct product of two cyclic groups. Suppose $G=\langle g,h \rangle \cong \langle g\rangle \times \langle h\rangle / \langle g^s h^{-t} \rangle$ for some $0<s<m, \ 0<t<n,$ where $m=|g|, \ n=|h|.$ Let $M(\lambda_1,\lambda_2,\eta_1,\eta_2)$ be a Majid algebra defined over $\Z_m\times\Z_n=\langle g'\rangle\times \langle h'\rangle$ with condition $\zeta_m^{a(s-1)}\lambda_1^s=\lambda_2^t, \ \zeta_{(m,n)}^{b(s-1)}\eta_1^s=\zeta_n^{c(t-1)}\eta_2^t.$ Then one may consider the Majid ideal of $M(\lambda_1,\lambda_2,\eta_1,\eta_2)$ generated by $g'^s-h'^t$ and a graded pointed Majid algebra $M$ generated by $G$ and two skew-primitive elements satisfying (R1-R2) can be obtained as the quotient Majid algebra $M(\lambda_1,\lambda_2,\eta_1,\eta_2)/I.$
\end{remark}

\section{Quasi-quantum groups of dimension $p^3$ and $p^4$}
In this section, we will give a classification of graded pointed Majid algebras with abelian coradical of dimension $p^3$ and $p^4$ for any prime number $p$ with a help of the classification results obtained in Section 2. Throughout the section, we say that a positive integer $N$ is the nilpotent order of a skew-primitive element $X$ if $X^{\vec{N}}=0$ while $X^{\overrightarrow{N-1}}\neq 0.$

\subsection{Pointed Majid algebras of dimension $p^3$}
With the assumption of dimension $p^3,$ we have the following list of graded pointed Majid algebras over the field $\k:$
\begin{itemize}
\item[1.] The Majid algebras $\k\langle g,X | gX = \zeta_p^r\zeta_{p^2}^aXg,\  g^p=1,\  X^{\vec{p^2}}=0 \rangle$ for all $0 \leq r < p$ with associator
$\Phi(g^i,g^j,g^k)=\zeta_p^{a[\frac{j+k}{p}]i}$ for $1\leq a <p$ and with  comultiplication determined by $\Delta(X)=X \otimes 1+g\otimes X$ and $\Delta(g)= g\otimes g.$

\item[2.] The Majid algebras $\k\langle g,h,X | gX=\zeta_p^uXg,\ hX=\zeta_p^vXh,\ g^p=h^p=1,\ gh=hg, \ X^{\vec{p}}=0 \rangle$ for all $1\leq u<p,\ 0\leq v <p$ with associator
$\Phi(g^ih^j,g^sh^t,g^kh^l)=\zeta_{p}^{b[\frac{k+s}{p}]j}\zeta_p^{c[\frac{t+l}{p}]j}$ for $0\leq b, \ c< p$ where either $b$ or $c$ is not 0, and with comultiplication determined by $\Delta(X)=X\otimes 1+ g\otimes X,\ \Delta(g)=g\otimes g,\ \Delta(h)=h\otimes h.$

\item[3.] The Majid algebras $\k\langle g,X | gX=\zeta_{p^2}^\alpha Xg,\ g^{p^2}=1,\ X^{\vec{p}}=0 \rangle$ for all $1\leq \alpha <p^2$ and $(\alpha, p)=1$ with associator
    $\Phi(g^i,g^j,g^k)=\zeta_{p^2}^{tp[\frac{j+k}{p^2}]i}$ for $1\leq t<p$ and with comultiplication determined by $\Delta(X)=X\otimes 1+ g^p\otimes X,\ \Delta(g)=g\otimes g.$
\end{itemize}

\begin{theorem}
If $M$ is a noncosemisimple graded pointed Majid algebra of dimension $p^3$ and is not twist equivalent to a Hopf algebra, then $M$ must be twist equivalent to one of the Majid algebras listed above.
\end{theorem}

\begin{proof}
First by assumption $M$ is not of the form $(\k G, \Phi).$ By $G$ we denote the set of group-like elements of $M$. If $|G|= 1,$ then in $M$ there are primitive elements (i.e., elements $X$ satisfying $\Delta(X)=X \otimes 1 + 1 \otimes X$) and $M$ contains the universal enveloping algebra $U(g)$ where $g$ is set of the primitive elements of $M,$ see \cite{qha1}. This contradicts with the finite dimensionality of $M.$ So $|G|\neq 1,$ and we have $|G|=p$ or $|G|=p^2$ by \cite[Theorem 3.2]{s}. Let $M^1$ denote the $\k$-space of all nontrivial $(1,g)$-primitive elements, that is, $M^1=\{ X \in M(1) | \Delta(X)=X\otimes 1+g\otimes X, \ \forall g \in G \}.$ Note that $M^1=\bigoplus_{g \in G}\ ^gM^1$ where $^gM^1=\{X \in M^1 | \Delta(X)=X\otimes 1+g\otimes X\}.$ In the following we split our discussion into several cases with respect to the order of $G.$

Consider first the case of $|G|=p.$ Choose any $g\in G$ such that $^gM^1 \neq 0.$ Clearly, $g \ne 1$ and $G=\langle g\rangle$ since the order of $G$ is a prime. Hence the associator of $M,$ up to twist equivalence, is determined by a 3-cocycle $\Phi_a(g^i,g^j,g^k)=\zeta_p^{a[\frac{j+k}{p}]i}$ for $1\leq a\leq p-1.$ By Subsection 2.2, we know that $^gM^1$ is a $(\k G,\widetilde{(\Phi_a)}_g)$-projective representation and it can be decomposed as direct sum of $1$-dimensional subrepresentations. Let $X\in \ ^gM^1$ such that $g\triangleright X=\lambda X.$ By Lemma 2.1, we have $\lambda^p=\zeta_p^a.$ Then apply Proposition 2.2 and Lemma 2.3, we see that the nilpotent order of $X$ is $|\zeta_p^a\lambda|=p^2.$ It follows that $M$ is actually generated by $G$ and $X$ and by a direct verification one can show that $M$ appears as one of the Majid algebras in 1 of the list.

Then we consider the case of $|G|=p^2$. It is well known that a group of order $p^2$ is either $\Z_p\times\Z_p$ or $\Z_{p^2}.$ Assume first $G=\Z_p\times \Z_p.$ As before choose $g\in G$ such that $^gM^1\neq 0$ and take another element $h\in G$ such that $G=\langle g\rangle \times \langle h\rangle.$ Then the associator of $M,$ up to twist equivalence, is determined by a 3-cocycle $\Phi(g^ih^j,g^sh^t,g^kh^l)=\zeta_p^{a[\frac{k+s}{p}]i}\zeta_{p}^{b[\frac{k+s}{p}]j}\zeta_p^{c[\frac{t+l}{p}]j}$ in which $a,b,c$ are not all 0. Take $X \in \ ^gM^1$ such that $g\triangleright X=\lambda_1 X, \ h\triangleright X=\lambda_2 X.$ Then $\lambda_1^p=\zeta_p^a$ by Lemma 2.1 and the nilpotent order of $X$ is $|\zeta_p^a \lambda_1|$ which is $p$ or $p^2.$ Taking the assumed dimension into account, then it is clear that $M$ is generated by $G$ and $X$ and the nilpotent order of $X$ has to be $p.$ So $|\zeta_p^a\lambda_1|=p$, and this implies $a=0$ and $\lambda_1=\zeta_p^i$ for some $0< i < p.$ In this situation $M$ appears as one of the Majid algebras in 2 of the list. Finally assume $G=\Z_{p^2}.$ Again choose a $g$ such that $^gM^1 \ne 0.$ If $|g|=p^2,$ then $G=\langle g \rangle$ and we can set the associator, up to twist equivalence, to be $\Phi_a(g^i,g^j,g^k)=\zeta_{p^2}^{a[\frac{j+k}{p^2}]i}$ for some $1 \leq a< p^2$. Let $X\in\ ^gM^1$ be nonzero such that $g\triangleright X=\lambda X.$ Then $\lambda^{p^2}=\zeta_{p^2}^a$ by Lemma 2.1 and the nilpotent order of $X$ is $|\zeta_{p^2}^a\lambda|>p^2.$ In this case, the Majid subalgebra of $M$ generated by $G$ and $X$ already has dimension $> p^3,$ which is absurd. This forces $|g|=p.$ Then we may assume $G=\langle h \rangle$ and $g=h^p$ and the associator, up to twist equivalence, is given by $\Phi_a(h^i,h^j,h^k)=\zeta_{p^2}^{a[\frac{j+k}{p^2}]i}$ for some $1 \leq a< p^2.$ In a similar manner, take a nonzero element $X\in \ ^gM^1$ such that $h\triangleright X=\lambda X$ with $\lambda^{p^2}=\zeta_{p^2}^{ap}.$ Then by the assumption of the dimension one sees that the nilpotent order of $X$ is $|\lambda^p|=p.$ It follows that $a=tp$ for $1\leq t<p$ and $\lambda=\zeta_{p^2}^\alpha$ for some $1\leq \alpha <p^2$ and $(\alpha,p)=1.$  In this situation $M$ appears as one of the Majid algebras in 3 of the list.

This completes the proof of the theorem.
\end{proof}

\begin{remark}
Pointed Hopf algebras of dimension $p^3$ over $\k$ were classified in \cite{as}. Combining those results with ours one may achieve a complete classification of graded pointed Majid algebras of dimension $p^3$ up to twist equivalence.
\end{remark}

\subsection{Pointed Majid algebras of dimension $p^4$}
First by direct construction we have the following list of graded pointed Majid algebras of dimension $p^4$ with abelian coradical.

\begin{itemize}
\item[1.] The Majid algebras $\k \langle g,h,X,Y |\ gX= \zeta_p^{\alpha} Xg,\ gY=\zeta_p^{-\alpha\beta }Yg,\ hX=\zeta_p^\gamma Xh,\ hY=\zeta_p^\eta Yh,\ XY=\zeta_p^{-\alpha\beta}YX,\ gh=hg,\ g^p=h^p=1,\ X^{\vec{p}}=Y^{\vec{p}}=0 \rangle$ for all $1\leq \alpha,\beta\leq p-1,\ 0\leq \gamma,\eta<p$ with associator
$\Phi(g^ih^j,g^sh^t,g^kh^l)=\zeta_{p}^{b[\frac{k+s}{p}]j}\zeta_p^{c[\frac{t+l}{p}]j}$ for $0\leq b,c< p$ where either $b$ or $c$ is nonzero, and with comultiplication determined by $\Delta(X)=X\otimes 1+g\otimes X,\ \Delta(Y)=Y\otimes 1+g^\beta \otimes Y, \ \Delta(g)=g\otimes g,\  \Delta(h)=h\otimes h.$

\item[2.] The Majid algebras $\k \langle g,h,X|gX=\zeta_p^\alpha \zeta_{p^2}^aXg,\ hX=\zeta_p^\beta Xh,\ g^p=h^p=1,\ gh=hg,\ X^{\vec{p^2}}=0 \rangle$ for all $0\leq \alpha, \ \beta< p$ with associator
$\Phi(g^ih^j,g^sh^t,g^kh^l)=\zeta_p^{a[\frac{k+s}{p}]i}\zeta_{p}^{b[\frac{k+s}{p}]j}\zeta_p^{c[\frac{t+l}{p}]j}$ for $0<a<p$ and $0\leq b,c<p,$
and with comultiplication determined by $\Delta(X)=X\otimes 1+ g\otimes X,\ \Delta(g)=g\otimes g,\ \Delta(h)=h\otimes h.$

\item[3.] The Majid algebras $\k \langle g,X|gX=\zeta_{p^2}^\alpha \zeta_{p^3}^aXg,\ g^{p^2}=1,\ X^{\vec{p^2}}=0\rangle$ for all $0 \leq \alpha < p^2$ with associator $\Phi(g^i,g^j,g^k)=\zeta_{p^2}^{a[\frac{j+k}{p^2}]i}$ for $0< a <p^2$ and $(a,p)=1,$
and with comultiplication determined by $\Delta(X)=X\otimes 1+ g^p\otimes X,\ \Delta(g)=g\otimes g.$

\item[4.] The Majid algebras $\k \langle g,X,Y |gX=\zeta_{p^2}^\alpha Xg,\ gY=\zeta_p^{\beta}\zeta_{p^2}^{-\alpha}Yg,\ XY=\zeta_p^{-\alpha}YX,\ g^{p^2}=1,\ X^{\vec{p}}=Y^{\vec{p}}=0\rangle$ for all $1\leq \alpha <p^2,$ $(\alpha,p)=1$ and $0\leq \beta< p$ with associator $\Phi(g^i,g^j,g^k)=\zeta_{p^2}^{tp[\frac{j+k}{p^2}]i}$ for $1\leq t<p$
and with comultiplication determined by $\Delta(X)=X\otimes 1+ g^p\otimes X,\ \Delta(Y)=Y\otimes 1+ g^p\otimes 1,\  \Delta(g)=g\otimes g.$

\item[5.] The Majid algebras $\k \langle g,h,X|gX=\zeta_p^\alpha Xg,\ hX=\zeta_p^\beta Xh,\ gh=hg,\ g^{p^2}=h^p=1,\ X^{\vec{p}}=0\rangle$ for all $1\leq \alpha< p$ and $0\leq \beta<p$ with associator $\Phi(g^ih^j,g^sh^t,g^kh^l)=\zeta_{p}^{b[\frac{k+s}{p^2}]j}\zeta_p^{c[\frac{t+l}{p}]j}$ for $0\leq b,c< p$ in which either $b$ or $c$ is nonzero, and with comultiplication determined by $\Delta(X)=X\otimes 1+ g\otimes X,\ \Delta(g)=g\otimes g,\ \Delta(h)=h\otimes h.$

\item[6.] The Majid algebras $\k \langle g,h,X|gX=\zeta_{p^2}^\alpha Xg,\ hX=\zeta_p^\beta Xh,\ gh=hg,\ g^{p^2}=h^p=1,\ X^{\vec{p}}=0\rangle$ for all $1\leq \alpha < p^2,$  $(\alpha,p)=1$ and $0\leq \beta <p$ with associator $\Phi(g^ih^j,g^sh^t,g^kh^l)=\zeta_{p}^{b[\frac{k+s}{p^2}]j}\zeta_p^{c[\frac{t+l}{p}]j}$ for $0\leq b,c< p$ in which either $b$ or $c$ is nonzero, and with comultiplication determined by $\Delta(X)=X\otimes 1+ g^p\otimes X,\ \Delta(g)=g\otimes g,\ \Delta(h)=h\otimes h.$
\item[7.] The Majid algebras $\k \langle g,h,X|gX=\zeta_{p^2}^\alpha Xg,\ hX=\zeta_p^\beta Xh,\ gh=hg,\ g^{p^2}=h^p=1,\ X^{\vec{p}}=0 \rangle$ for all $1\leq \alpha < p^2,$  $(\alpha ,p)=1$ and  $0\leq \beta <p$ with associator $\Phi(g^ih^j,g^sh^t,g^kh^l)=\zeta_{p^2}^{tp[\frac{k+s}{p^2}]i}\zeta_{p}^{b[\frac{k+s}{p^2}]j}\zeta_p^{c[\frac{t+l}{p}]j}$ for $1\leq t<p$ and  $0\leq b,c< p$,
    and with comultiplication determined by $\Delta(X)=X\otimes 1+ g^p\otimes X,\ \Delta(g)=g\otimes g,\ \Delta(h)=h\otimes h.$

\item[8.] The Majid algebras $\k \langle g,h,X|gX=\zeta_{p^2}^\alpha\zeta_{p^3}^b Xg,\ hX=\zeta_{p}Xh,\ gh=hg,\ g^{p^2}=h^{p}=1,\ X^{\vec{p}}=0 \rangle$ for all $0\leq \alpha< p^2$ with associator $\Phi(g^ih^j,g^sh^t,g^kh^l)=\zeta_{p^2}^{a[\frac{k+s}{p^2}]i}\zeta_{p}^{b[\frac{k+s}{p^2}]j}$ for $0\leq a\leq p^2,\ 0\leq b< p$ in which either $a$ or $b$ is nonzero, and with comultiplication determined by $\Delta(X)=X\otimes 1+ h\otimes X,\ \Delta(g)=g\otimes g,\ \Delta(h)=h\otimes h.$

\item[9.] The Majid algebras $\k \langle e,f,g,X| eX=\zeta_p^\alpha Xe,\ fX=\zeta_p^{\beta}Xf,\ gX=\zeta_p^\gamma Xg,\ ef=fe,\ eg=ge,\ fg=gf,\ X^{\vec{p}}=0\rangle$ for all $1\leq \alpha <p$ and $0\leq \beta,\gamma <p$ with associator \[ \quad \quad \quad \Phi(e^{i_1}f^{i_2}g^{i_3},e^{j_1}f^{j_2}g^{j_3},e^{k_1}f^{k_2}g^{k_3}) =\zeta_p^{\sum_{l=1}^3 a_li_l[\frac{j_l+k_l}{p}]}\zeta_p^{a_4i_2[\frac{j_1+k_1}{p}]+a_5i_3[\frac{j_1+k_1}{p}]+a_6i_3[\frac{j_2+k_2}{p}]}\zeta_p^{a_7k_1j_2i_3}\]
in which $a_1=0, \ 0\leq a_i<p, \ \forall 2\leq i\leq 7$ and at least one $a_i$ is nonzero, and with comultiplication determined by $\Delta(e)=e\otimes e,\ \Delta(f)=f\otimes f,\ \Delta(g)=g\otimes g,\ \Delta(X)=X\otimes 1+ e\otimes X.$
\end{itemize}

\begin{theorem}
If $M$ is a noncosemisimple graded pointed Majid algebra of dimension $p^4$ with abelian coradical and is not twist equivalent a Hopf algebra, then $M$ is twist equivalent to one of the Majid algebras listed above.
\end{theorem}
\begin{proof}
Let $G$ be the group of group-like elements of $M.$ By a similar argument as used in the beginning of the proof of Theorem 3.1, we have $|G|=p, \ p^2 \ or \ p^3.$ Again by $M^1$ we denote the space spanned by $(1,g)$-primitive elements of $M$ and write $M^1=\bigoplus_{g\in G}\ ^gM^1.$ We split the discussion into 5 cases with respect to the structure of $G.$

(1) $G=\Z_p.$ Choose any $g$ such that $^gM^1 \neq 0.$ Then clearly $G=\langle g\rangle.$ Hence the associator, up to twist equivalence, is given by $\Phi_a(g^i,g^j,g^k)=\zeta_p^{a[\frac{j+k}{p}]i}$ for some $1\leq a<p.$ Take a nonzero $X\in\ ^gM^1$ such that $g\triangleright X=\lambda X$ with $\lambda^p=\zeta_p^a.$ Then the nilpotent order of $X$ is $|\zeta_p^a\lambda|=p^2.$ Considering the dimension of $M,$ one observes that there must be another skew-primitive element $Y$ which is linearly independent to $X.$ By a similar discussion, we know the nilpotent order of $Y$ is also $p^2.$ In this situation, the Majid subalgebra of $M$ generated by $g, X, Y$ already has dimension $>p^4.$ Hence there are no such Majid algebras with coradical of dimension $p.$

(2) $G=\Z_p\times \Z_p.$ Choose any $g$ such that $^gM^1 \neq 0$ and a nonzero $X \in\ ^gM^1.$ Then clearly $|g|=p$ and the nilpotent order of $X$ is either $p$ or $p^2.$ First assume the latter. Then we can find another element $h \in G$ such that $G=\langle g\rangle \times \langle h\rangle$ and the associator, up to twist equivalence, is given by $\Phi(g^ih^j,g^sh^t,g^kh^l)=\zeta_p^{a[\frac{k+s}{p}]i}\zeta_{p}^{b[\frac{k+s}{p}]j}\zeta_p^{c[\frac{t+l}{p}]j}$ for some $0\leq a,b,c<p$ in which $a \ne 0.$ In this situation, $M$ appears as one of the Majid algebras in 2 of the list. Next assume the nilpotent order of $X$ is $p.$ Then in $M^1$ there must be another skew-primitive element $Y$ which is linearly independent to $X.$ Assume $\Delta(Y)=Y \otimes 1 + f \otimes Y.$ If $f=g^\alpha$ for some $1 \leq \alpha<p,$ then we can choose $h \in G$ such that $G=\langle g\rangle \times \langle h\rangle$ and assume the associator as before.  Assume $g\triangleright Y=\beta Y$, then we have $\beta^p=1$ and $fY=\beta^\alpha Yf.$ Then the nilpotent order of $Y$ is $|\beta^\alpha|=p.$ In this situation, the Majid algebra $M$ is generated by $G, X, Y$ and appears as one of those in 1 of the list. If $f \notin \langle g \rangle,$ then $G=\langle g\rangle \times \langle f \rangle$ and the associator of $M$ can be given as above with $h$ replaced by $f.$ By taking the dimension of $M$ into account, the nilpotent order of $Y$ is $p$ and $X$ and $Y$ are skew-commutative. This forces $a=c=0$ and $b \ne 0$ in the associator. Note that such $M$ satisfies the condition of those in 1 of Theorem 2.5 with $m=n=p.$ However, it is quite obvious that the equation (2.29) has no solutions if $0<b<p,$ therefore in this situation there are no Majid algebras satifying the required condition.

(3) $G=\Z_{p^2}.$ Choose any $g$ such that  $^gM^1\neq 0$, then $|g|=p$ or $p^2.$  Firstly assume $|g|=p^2,$ then  $G=\langle g \rangle$ and the associator, up to twist equivalence, is given by $\Phi(g^i,g^j,g^k)=\zeta_{p^2}^{a[\frac{j+k}{p^2}]i}$ for some $0<a<p^2$. Take a nonzero element $X\in \ ^gM^1$  such that $g\triangleright X=\lambda X$ with $\lambda^{p^2}=\zeta_{p^2}^{a}.$  Then the nilpotent of $X$ is $|\zeta_{p^2}^a\lambda|>p^2,$ so the Majid subalgebra of $M$ generated by $g, X$ already has dimension $>p^4.$ This is not possible. Hence in this situation there are no Majid algebras with the required condition. Next assume $|g|=p.$ Then $G=\langle h\rangle $ for some $h$ such that $g=h^p.$ The associator may be given by $\Phi(h^i,h^j,h^k)=\zeta_{p^2}^{a[\frac{j+k}{p^2}]i}$ for some $0<a<p^2$ up to twist equivalence. Choose a nonzero $X\in\ ^gM^1$ such that $h\triangleright X=\lambda X$ with $\lambda^{p^2}=\zeta_{p^2}^{ap}.$ Then $g\triangleright X=h^p\triangleright X=\zeta_{p^2}^{ap(p-1)}\lambda^pX,$ so $gX=\lambda^pXg$ and the nilpotent order of $X$ is $|\lambda^p|.$ If $(a, p)=1,$ then $|\lambda^p|=p^2.$ In this situation $M$ is generated by $h, X$ and appears as one of the Majid algebras in 3 of the list. If $a=tp$ for some $1\leq t<p,$ then $|\lambda^p|=p.$ So there must be another skew-primitive element $Y$ of nilpotent order $p$ which is linearly independent to $X,$ and $X,\ Y$ are skew-commutative. In this case $M$ appears as one of those in 4 of the list.

(4) $G=\Z_{p^2}\times \Z_p.$ Choose $^gM^1\neq 0$ and a nonzero $X\in\ ^gM^1.$ The assumption of $\dim M=p^4$ forces that the nilpotent order of $X$ is $p.$ If $|g|=p^2,$ then we can find another $h\in G$ such that $G=\langle g\rangle \times \langle h\rangle$ and the associator, up to twist equivalence, is given by $\Phi(g^ih^j,g^sh^t,g^kh^l)=\zeta_{p^2}^{a[\frac{k+s}{p^2}]i}\zeta_{p}^{b[\frac{k+s}{p^2}]j}\zeta_p^{c[\frac{t+l}{p}]j}$ for some $0\leq a<p^2, 0<b,c<p$ and at least one of them is nonzero. Let $g \triangleright X=\lambda X$ with $\lambda^{p^2}=\zeta_{p^2}^a.$ The nilpotent order of $X$ is $|\zeta_{p^2}^a\lambda|=p,$ it follows that $a=0$ and $\lambda=\zeta_p^\alpha,$ for some $1\leq \alpha <p.$ In this situation, $M$ appears as one of the Majid algebras in 5. If $|g|=p,$ then there are two cases to be considered. Firstly if there is an $f \in G$ of order $p^2$ such that $g=f^p,$ then we can find another element $h\in G$ such that $G=\langle f\rangle \times \langle h\rangle$ and the associator, up to twist equivalence, can be given as above with $g$ replaced by $f.$ Let $f \triangleright X=\lambda X$ with $\lambda^{p^2}=\zeta_{p^2}^{ap}.$ Then $gX=\lambda^pXg$ and $|\lambda^p|=p$ as the nilpotent order of $X$ is $p.$ It follows that $a=0$ or $tp$ for $1\leq t<p$ and $\lambda=\zeta_{p^2}^\alpha$ for some $1\leq\alpha <p^2$ and $(\alpha,p)=1.$ If $a=0,$  $M$ appears as one of the Majid algebras in 6 listed above; if $a=tp$ for $1\leq t<p,$ then $M$ appears as one of the Majid algebras in 7. Secondly if there is not an $f$ of order $p^2$ such that $g = f^p,$ then $G=\langle h \rangle \times \langle g\rangle$ for some $h$ of order $p^2$ and the associator, up to twist equivalence, is given as above in a similar manner. Let $g\triangleright X=\lambda X$ with $\lambda^{p}=\zeta_p^c.$ Again by the condition that the nilpotent order of $X$ is $p,$ it follows that $c=0$ and $\lambda=\zeta_p^\alpha$ for some $1\leq \alpha <p.$ In this situation $M$ appears as one of the Majid algebras in 8.

(5) $G=\Z_p\times \Z_p \times \Z_p.$ Assume that $^eM^1\neq 0,$ then $|e|=p$ and we can find elements $f,g\in G$ such that $G=\langle e\rangle\times \langle f\rangle\times \langle g\rangle.$ Then according to \cite{bgrc2}, the associator of $M,$ up to twist equivalence, is of the form
\begin{equation*}
\Phi_{\overline{a}}(e^{i_1}f^{i_2}g^{i_3},e^{j_1}f^{j_2}g^{j_3},e^{k_1}f^{k_2}g^{k_3})
=\zeta_p^{\sum_{l=1}^3 a_li_l[\frac{j_l+k_l}{p}]}\zeta_p^{a_4i_2[\frac{j_1+k_1}{p}]+a_5i_3[\frac{j_1+k_1}{p}]+a_6i_3[\frac{j_2+k_2}{p}]}\zeta_p^{a_7k_1j_2i_3},
\end{equation*}
where $\overline{a}\in\{(a_1,\cdots ,a_7)|0\leq a_i <p,\ 1\leq i\leq 7\}$ and at least one of the arguments of $\overline{a}$ is nonzero. Take a nonzero $X \in \ ^eM^1$ such that $e\triangleright X=\lambda X$ with $\lambda^p=\zeta_p^{a_1}.$ Then the nilpotent order of $X$ is $|\zeta_p^{a_1}\lambda|=p$ by the assumption $\dim M=p^4.$ This forces $a_1=0$ and $\lambda=\zeta_p^\alpha$ for some $1\leq \alpha<p.$ In this situation $M$ appears as one of the Majid algebras in 9 of the list.
\end{proof}

\noindent{\bf Acknowledgement:} The authors are very grateful to the referee for the valuable comments and suggestions which helped to improve the exposition.

\end{document}